\documentclass[a4paper,10pt,leqno]{article}
\usepackage{geometry}
\usepackage[latin1]{inputenc} 

\usepackage{amsmath,amsfonts,amssymb}
\usepackage{dsfont}

\usepackage{amsthm}

\usepackage[all]{xy}
\usepackage{mathrsfs}
\newtheorem{theorem}{Theorem}
\newtheorem{lemma}{Lemma}
\newtheorem{example}{Assumptions}
\newtheorem{prop}{Proposition}
\usepackage{graphicx}
\usepackage{stackrel}
\usepackage[english]{babel} 
\usepackage{xcolor}

\theoremstyle{definition}

\newtheorem{definition}{Definition}[section]
\numberwithin{equation}{section}

\newtheorem{Rem}{Remark}[section]
\title{Local well-posedness for a class of singular Vlasov equations}
\author{Thomas Chaub}
\date{}

\newcommand{\norm}[1]{{\left\lVert#1\right\rVert}}

\newcommand{\intr}[1]{\int_{\bf R^{n}}}

\newcommand{\intrd}[1]{\int_{\bf R^{2n}}}

\newcommand{\intt}[1]{\int_{\bf T^{n}}}

\begin{document}
\renewcommand{\proofname}{Proof}
\maketitle
\renewcommand{\abstractname}{Abstract}
\renewcommand{\contentsname}{Summary}

\begin{abstract} 
In this article we study a singular Vlasov system on the torus where the force field has the smoothness of a (fractional) derivative $D^{\alpha}$ of the density, where $\alpha>0$. We prove local well-posedness in Sobolev spaces without restriction on the data. This is in sharp contrast with the case $\alpha=0$ which is ill-posed in Sobolev spaces for general data.  
\end{abstract}

\bigbreak
\tableofcontents

\bigbreak

\section{Introduction and main result}

\subsection{Context and notations}

In this article we focus on a class of nonlinear singular Vlasov systems in the torus $\mathbb{T}^d=\mathbb{R}^d/2\pi \mathbb{Z}^d$, a prototypal example is 
\begin{equation}\label{SystRef2}
 \left\{
    \begin{array}{ll}
        \partial_t{f}+v\cdot \nabla_x f + E  \cdot \nabla_v f  = 0,\\
        \gamma E =-\nabla_x V, \\
	(-\Delta_x)^{\alpha/2}V= \left(\rho-\rho_0\right), \\
	f (0,x,v) = f^0 (x,v),
    \end{array}
\right.
\end{equation}

where the parameter $\gamma$ will be taken to $1$ or $-1$, and the equation will be respectively called repulsive or attractive.  The function $f$ stands for a distribution function in the domain $\mathbb{T}^d\times\mathbb{R}^d$ and may represent the distribution of electrons in a plasma, or the density of stars in stellar dynamics. The density associated with $f$ will be denoted $\rho=\int_{\mathbb{R}^d}f(t,x,v) dv$. The term $\rho_0=\int_{\mathbb{R}^d\times \mathbb{T}^d}f^0(t,x,v) dvdx$ corresponds, in the repulsive setting, to the density of ions. In the attractive setting, we remove the mean value of $\rho$ as a trick called the Jeans swindle, which is mathematically relevant \cite{Swindle}. The initial condition, $f^0$, can  be taken non-negative, and $\alpha>0$ is a parameter. \\
The aim of this work is to give a local well-posedness theory in Sobolev regularity for the general case $\alpha>0$. To introduce our main result, we shall first introduce some notations.

\bigbreak

\bigbreak
For $k\in \mathbb{N}$, $r\in \mathbb{N}$, we introduce the weighted Sobolev norms 
\begin{equation*}
\norm{f}_{\mathcal{H}_{r}^{k}}:=\left(\sum_{\vert\alpha\vert+\vert\beta\vert\leq k}\int_{\mathbb{T}^{d}}\int_{\mathbb{R}^{d}} (1+\vert v \vert ^{2} )^{r} \vert \partial_{x}^{\alpha}\partial_{v}^{\beta}f\vert^{2}dvdx   \right)^{1/2},
\end{equation*}
where, for $\alpha=(\alpha_{1},..., \alpha_{d}), \beta=(\beta_{1},...,\beta_{d})\in\mathbb{N}^{d}$, we write 
\begin{equation*}
\vert\alpha\vert =\sum_{i=1}^{d} \alpha_{i},\quad \vert\beta\vert =\sum_{i=1}^{d} \beta_{i},
\end{equation*}

\begin{equation*} 
\partial_{x}^{\alpha}:= \partial_{x_{1}}^{\alpha_{1}}...\partial_{x_{d}}^{\alpha_{d}}, \quad \partial_{x}^{\beta}:= \partial_{x_{1}}^{\beta_{1}}...\partial_{x_{d}}^{\beta_{d}}.
\end{equation*}

We will also use the classic Sobolev spaces. We will write $H_{x,v}^{k}$ (resp $W^{k,\infty}_{x,v})$ the standard Sobolev space for the norm $L^{2}$ (resp $L^{\infty}$) for functions depending on  $(x,v)$, and  $H_{x}^{k}$ for functions only depending on $x$. 
\bigbreak

\bigbreak
We will use the following notation and convention for the Fourier transform of a function $U$ 
\begin{equation*}
\widehat{U} (k) :=(2\pi)^{-d}\int_{\mathbb{T}^d} U(x) e^{-ik\cdot x} dx.
\end{equation*}
\bigbreak

\bigbreak

Instead of (\ref{SystRef2}), we consider a more general system under the form 
\begin{equation}\label{SystRef}
 \left\{
    \begin{array}{ll}
        \partial_t{f}+v\cdot \nabla_x f + E  \cdot \nabla_v f  = 0,\\
        E =-\nabla_x \int_{\mathbb{T}^d} U(x-y) \left(\rho(y)-\rho_0\right)dy, \\
	f (0,x,v) = f^0 (x,v),
    \end{array}
\right.
\end{equation}
making the following assumptions for $U$ :
\begin{example} [A1]
$\widehat{U}(0)=0$ and there exists $\alpha>0$, $C_{\alpha}>0$ such that 
\begin{equation}\label{ass}
\forall k\neq 0, \:k\in\mathbb{Z}^d, \quad\vert \widehat{U} (k)\vert \leq \frac{C_{\alpha}}{\vert k\vert^{\alpha}}.
\end{equation}
\end{example}
In particular, by taking $\widehat{U}(k)=\pm \frac{1}{\vert k \vert ^{\alpha}}$, $\widehat{U}(0)=0$, we recover (\ref{SystRef2}).\\
Note that we can also consider potentials $U=\displaystyle \sum_{i=1}^{p} U_i$ such that each $U_i$ satisfies (\ref{ass}) for $\alpha_i>0$. Then $U$ satisfies (\ref{ass}) with $\alpha=\min(\alpha_i)_{i\in[1,p]}$. Those types of potentials may appear in some physical models in gravitation (see the discussion below about Manev potentials).\bigbreak
We denote by $\lfloor x \rfloor$ the floor value of $x$ and
\begin{equation}
m_{0}=3+\frac{d}{2}+p_{0}, \quad p_{0}=\lfloor \frac{d}{2} \rfloor+1, \quad r_{0}=\max\left(d,2+\frac{d}{2}\right).
\end{equation}
Our main well-posedness result is the following

\begin{theorem}\label{MT}
Let $f^0\in \mathcal{H}^{m}_{2r}$ with $m>m_0$, $2r>r_0$. Assuming (A1), then there exists $T>0$ for which there exists a unique solution of the system (\ref{SystRef}) with initial condition $f^0$ and such that $f \in\mathcal{C}([0,T],\mathcal{H}^{m-1}_{2r})$, $\rho \in L^{2}([0,T],H_x^{m})$.
\end{theorem}
Note that we consider both the attractive and repulsive case.

\begin{Rem}
Notice that in the above statement the solution $f$ is less regular than the initial condition $f^0$. This was expected because of the nature of the equation. Nevertheless, the fact that $f$ is in $\mathcal{C}([0,T],\mathcal{H}^{m-1}_{2r})$ is not optimal. With sharper estimates in the following proof, when $\alpha\leq1$, $f$ is expected to be in $\mathcal{C}([0,T],\mathcal{H}^{m-1+\alpha}_{2r})$, provided that we use the fractional space $\mathcal{H}^{s}_{2r}$ when $s$ is not an integer (which has not been defined here). For simplicity, we will only give the proof of Theorem \ref{MT}, as the gain of regularity we could obtain does not seem fundamental. 
\end{Rem}
Our results could be easily extended in the case of $\mathbb{R}^d$ instead of $\mathbb{T}^d$ assuming that
\begin{equation*}
 \forall \xi \in \mathbb{R}^d, \: \vert \widehat{U}(\xi)\vert\leq \frac{C_{\alpha}}{1+\vert\xi\vert^{\alpha}}.
\end{equation*}
\bigbreak

The main interest of our theorem is for small values of $\alpha$. Indeed, if we take $\alpha=2$ in (\ref{SystRef2}), the system we obtain is the well-known Vlasov-Poisson system, and a local well-posedness theory is very easy to get. We even have global existence for physical dimensions $d\leq 3$ (see \cite{Ukai}, \cite{Pf}, \cite{Sch}).

\bigbreak
For $1\leq \alpha <2$, system (\ref{SystRef2}) behaves at least like a Burgers type equation and local well-posedness theory follows by standard energy estimates.
The case $\alpha=1$ in particular, has physical meaning. It was first introduced by Manev as a correction of the Newtonian potential. The interactions between Newtonian and Manev potentials have been studied (see \cite{Vl-Ma}), but the Manev potential on its own is also interesting. In the attractive setting, it is called the Pure Stellar dynamic Manev system (PSM) \cite{VMI} , and it arises in stellar dynamics.\\
For $\frac{3}{4}\leq\alpha\leq 1$ local well-posedness was recently obtained in (\ref{SystRef2}) using the additional regularity provided by averaging lemmas in the whole space in \cite{A-L2022}. In addition, they give general conditions such that the system (\ref{SystRef2}), that they call the \textit{Vasov-Riesz system} because of the introduction of a Riesz type interaction, have finite-time singularity formation for solutions.

\bigbreak

Thus, the main contribution of this article corresponds to the case $\alpha\in(0,\frac{3}{4})$. The well-posedness theory of the equation is more challenging, because the apparent lost derivative in $x$ on the force field $E$ is no longer compensated by regularization. In the critical case $\alpha=0$ (which does not enter our framework), named the Vlasov-Dirac-Benney system, it has been proved that the system is in general ill-posed in Sobolev spaces \cite{IPVB}. Nevertheless, on the Torus and in the repulsive case, we can ensure hypotheses to have the local well-posedness of the system. The first results were obtained in dimension $d=1$ by Bardos and Besse \cite{B-B}. They proved that the Vlasov-Dirac-Benney system is locally well-posed provided that for all $(t,x)$, there exists a function $m(t,x)$ such that $v\mapsto f(t,x,v)$ remains compactly supported and is increasing for $v\leq m(t,x)$, decreasing for $v\geq m(t,x)$, in other words, \textit{one bump} shaped functions. More recently, Han-Kwan and Rousset \cite{HKR} have proved that the Vlasov-Dirac-Benney system is locally well-posed in any dimension for Sobolev regularity, provided that for all $x$, the profile $v\mapsto f^0(x,v)$ satisfies a Penrose stability condition, which means that if we define the Penrose function 
\begin{equation}
\mathcal{P}(\gamma,\tau,\eta,f)=1-\int_{0}^{+\infty} e^{-(\gamma+i\tau)s} \frac{i\eta}{1+\vert\eta\vert^2}\cdot (\mathcal{F}_v\nabla_v f) (\eta s) ds,\quad \gamma>0, \tau \in \mathbb{R}, \eta \in \mathbb{R}^d \backslash\{0\},
\end{equation}
$v\mapsto f^0(x,v)$ must satisfy the following condition, for some $c_{0}>0$
\begin{equation}
\underset{(\gamma,\tau,\eta)\in(0,+\infty)\times \mathbb{R}\times \mathbb{R}^d\backslash\{0\}}{\inf} \vert \mathcal{P}(\gamma,\tau,\eta,f^0)\vert \geq c_0.
\end{equation}
\bigbreak

\subsection{Sketch of the proof}
In order to prove Theorem \ref{MT}, we shall define another system, which is a regularized Vlasov-Poisson type system, that depends on a new parameter $\varepsilon>0$ 
\begin{equation}\label{SystP}
 \left\{
    \begin{array}{ll}
        \partial_t{f_\varepsilon }+v\cdot \nabla_x f_\varepsilon +E_\varepsilon  \cdot \nabla_v f_\varepsilon  = 0,\\
        E_\varepsilon =-\nabla_{x}(U* V_\varepsilon),  \\
	-\varepsilon^{2}\Delta_x V_\varepsilon  +V_\varepsilon =\int_{\mathbb{R}^{d}} f_\varepsilon (.,v) dv , \\
	f_\varepsilon (0,x,v) = f^0_\varepsilon (x,v).
    \end{array}
\right.
\end{equation}
  
Taking formally the limit $\varepsilon \rightarrow 0$, we obtain the system (\ref{SystRef}), which is exactly the system we want to study. The idea is to use the well-posedness of the system (\ref{SystP}) to get a family of functions $(f_\varepsilon)$ which satisfy uniform estimates. Then, using compactness  we extract (in a certain sense that will be defined later) a function $f$ which is a solution of the limit system (\ref{SystRef}). After that, we will show uniqueness for the solution of (\ref{SystRef}) in the class of $\{ f \in\mathcal{C}([0,T],\mathcal{H}^{m-1}_{2r})$, $\rho \in L^{2}([0,T],H_x^{m})$\}.\\The first step is to show that we can get a family of functions $(f_\varepsilon)$ defined on a time interval $[0,T]$ with $T$ being independent of $\varepsilon$. To do so, we introduce the following key quantity  
\begin{equation}
\mathcal{N}_{m,2r}(t,f):=\norm{f}_{L^{\infty}\left([0,t],\mathcal{H}^{m-1}_{2r}\right)}+\norm{\rho}_{L^{2}\left([0,t],H_x^{m}\right)},
\end{equation}
and we have the following theorem 

\begin{theorem}\label{theorem2}
Let $\alpha \in \mathbb{R^*_+}$. We assume that for all $\varepsilon \in(0,1]$, $f^{0}_{\varepsilon}\in\mathcal{H}^{m}_{2r}$ with $m>m_{0}$, $2r>r_{0}$ and that there exists $M_{0}>0$ so that for all $\varepsilon\in (0,1]$, $\norm{f^{0}_{\varepsilon}}_{\mathcal{H}_{m}^{2r}}\leq M_{0}$. \\
Then there exist $T>0$, $R>0$ (independent of $\varepsilon$) and a unique solution  $f_{\varepsilon}\in \mathcal{C}([0,T],\mathcal{H}^{m}_{2r})$ of (\ref{SystP}) with 
\begin{equation}
\underset{\varepsilon\in(0,1]}{\sup}\mathcal{N}_{m,2r}(T,f_{\varepsilon})\leq R.
\end{equation}
\end{theorem}
\bigbreak
The proof of this theorem will rely on a bootstrap argument following the strategy of \cite{HKR}. For most of the proof, we will try to estimate the key quantity $\mathcal{N}_{m,2r}(t,f_{\varepsilon})$ independently of $\varepsilon$. Of course, if we could give an estimate of $\norm{f_{\varepsilon}}_{L^{\infty}\left([0,t],\mathcal{H}^{m}_{2r}\right)}$ independently of $\varepsilon$, the proof would be over, but this cannot be done in the general case. This is the reason why we choose to work with $\norm{f_{\varepsilon}}_{L^{\infty}\left([0,t],\mathcal{H}^{m-1}_{2r}\right)}$ (losing one derivative for $f$ is enough to give an estimate that does not depend on $\varepsilon$), and $\norm{\rho_{\varepsilon}}_{L^{2}\left([0,t],H_x^{m}\right)}$. The fact that it is possible to give an estimate independent of $\varepsilon$ on $\rho$ without loss of derivatives while it is not for $f$ is not trivial and is fundamental here. This is reminiscent of the results on averaging lemmas (see \cite{AL}). Here though, the situation is a bit different, as we will be averaging in $v$ but also in time. More precisely, anticipating a bit on the following proof, let us look at the operator $K$ (defined in Lemma \ref{AOP}),
\begin{equation*}
K_{G}(F)(t,x) =\int_{0}^{t}\int (\nabla_{x}F)(s,x-(t-s)v)\cdot G(t,s,x,v) dvds.
\end{equation*}
There is an apparent loss of one derivative over $x$. However, the key proposition \ref{MdK} shows that $K_G$  is a bounded operator from $L^{2}([0,T],H^{\alpha}_{x})$ into $L^{2}([0,T],L^{2}_{x})$ (provided that $G$ is smooth enough and $\alpha\in(0,1)$), and the bound can be estimated up to a constant by $T^{\alpha/2}$ (actually, the result remains true for $\alpha=0$ \cite{HKR}). The apparent loss of a derivative in $x$ is somehow compensated by taking the averages in $v$ and in time.\\

\begin{Rem}
Because of the apparent links with the theory of averaging lemmas \cite{AL}, we could try to apply directly those results to prove our theorem. As stated earlier, a recent work \cite{A-L2022} has proved that, when $x$ is in the whole space, the system (\ref{SystRef2}) is locally well-posed up to $\alpha=3/4$. Nevertheless, standard averaging lemmas cannot cover the whole range $\alpha\in(0,1]$.
\end{Rem}
Sections 2, 3 and 4 will be dedicated to the proof of Theorem 2. In section 5, we give the elements to conclude the proof of the main Theorem \ref{MT}.

\section{Beginning of the proof of Theorem 2}
\bigbreak

\subsection {Preliminary lemmas}
We present here some lemmas that will be useful for the proof of Theorem \ref{theorem2}. They have been proved in \cite{HKR}. We use the notation $[A,B]=AB-BA$  to denote the commutator between two operators.
\begin{lemma}

Let $s \geq 0$ and $\chi = \chi(v)$ a non-negative smooth function, with $\vert \partial^{\alpha}\chi\vert \leq C_{\alpha}\chi$ for all $\alpha\in\mathbb{N}^{d}$, $\vert \alpha\vert \leq s $. \smallbreak
\begin{itemize}
\item Consider two functions $f=f(x,v)$, $g=g(x,v)$, then for all $k \geq s/2$ 

\begin{equation} 
\norm{\chi f g}_{H^{s}_{x,v}} \lesssim  \norm{\chi g}_{H^{s}_{x,v}}\norm{f}_{W^{k,\infty}_{x,v}}+  \norm{\chi f}_{H^{s}_{x,v}}\norm{g}_{W^{k,\infty}_{x,v}}.
\end{equation}
\item Consider a function $E=E(x)$ and a function $F=F(x,v)$. Then for all $s_{0} > d$

\begin{equation} 
\norm{\chi E F}_{H^{s}_{x,v}} \lesssim  \norm{E}_{H^{s_{0}}_{x}}\norm{\chi F }_{H^{s}_{x,v}}+  \norm{\chi F}_{H^{s}_{x,v}}\norm{E}_{H^{s}_{x}}.
\end{equation}
\item Consider a vector field $E=E(x)$ and a function $f=f(x,v)$, then for all $s_{0} > 1+d$ and for all $\alpha,$ $\beta \in \mathbb{N}^{d}$ with $\vert \alpha \vert +\vert \beta \vert =s \geq 1$

\begin{equation} \label{Lemma}
\norm{\chi [\partial^{\alpha}_{x} \partial^{\beta}_{v}, E(x)\cdot \nabla_{v}]f}_{L^{2}_{x,v}} \lesssim   \norm{E}_{H^{s_{0}}_{x}}\norm{\chi f }_{H^{s}_{x,v}}+  \norm{\chi f}_{H^{s}_{x,v}}\norm{E}_{H^{s}_{x}}. 
\end{equation}
\end{itemize}
\end {lemma}

The main use of this lemma will be for $\chi (v) = (1+\vert v \vert ^{2})^{\frac{\sigma}{2}}$ thus yielding estimates for products in the space $\mathcal{H}^{s}_{\sigma}$.\\
Let us introduce another product type estimate. 
\begin{lemma} 
Consider two functions $f=f(x,v)$, $g=g(x,v)$. Then, for all $s \geq 0$, $\alpha, \beta \in \mathbb{N}^{2d}$ with $\vert \alpha \vert + \vert \beta \vert \leq s $, and $\chi(v)$ positive function satisfying $\vert \partial^{\alpha}\chi\vert \leq C_{\alpha}\chi$ (as in the precedent lemma). We have

\begin{equation}
\norm{\partial^{\alpha}_{x,v} f \partial^{\beta}_{x,v} g}_{L^{2}_{x,v}} \lesssim \norm{\frac{1}{\chi} f}_{L^{\infty}_{x,v}}   \norm{\chi g}_{H^{s}_{x,v}}+\norm{\frac{1}{\chi} f}_{H^{s}_{x,v}} \norm{\chi g}_{L^{\infty}_{x,v}}.
\end{equation}

\end{lemma}

Finally, the following lemma is a very useful commutation formula between $\partial_{x}^{\alpha}\partial_{v}^{\beta}$ and the transport operator $\mathcal{T}$ defined by 

\begin{equation}
\mathcal{T} =\partial_{t}+v \cdot \nabla_{x}+E \cdot \nabla_{v}.
\end{equation}

\bigbreak

\begin{lemma}
For every $\alpha,$ $\beta \in \mathbb{N}^{d}$, we have for every smooth functions $f$

\begin{equation}\label{ComLem}
\partial_{x}^{\alpha}\partial_{v}^{\beta}(\mathcal{T}f) = \mathcal{T}(\partial_{x}^{\alpha}\partial_{v}^{\beta}f) + \sum_{i=1}^{d}\mathds{1}_{\beta_{i}\ne 0}\partial_{x_{i}}\partial_{x}^{\alpha}\partial_{v}^{\overline{\beta}^{i}}f+[\partial_{x}^{\alpha}\partial_{v}^{\beta},E\cdot \nabla_{v}]f,
\end{equation}
where $\overline{\beta}^{i}$ is equal to $\beta$ except $\overline{\beta}^{i}_{i}=\beta_{i}-1$.

\end{lemma}

\subsection{Setting up the bootstrap} 
The proof of the theorem will rely on a bootstrap argument. From standard energy estimates (that we shall recall later, see the proof of proposition \ref{PE}), we first have the following result.

\begin{prop} 
The system (\ref{SystP}) is locally well-posed in $\mathcal{H}^{m}_{2r}$ for all $m$ and $r$ satisfying $m > 1+d$ and $2r > d/2$. In other words, if $f_{\varepsilon}^{0} \in \mathcal{H}^{m}_{2r}$, there exists $T>0$ (which depends on $\varepsilon$) so that there exists a unique $f_{\varepsilon} \in \mathcal{C}([0,T],\mathcal{H}^{m}_{2r})$ solution of system (\ref{SystP}). 

\end{prop}

Thanks to the previous proposition, we may define a maximal solution $f_{\varepsilon}\in \mathcal{C}\left([0,T^{*}),\mathcal{H}^{m}_{2r}\right)$. As a direct consequence, for every $T<T^{*},\quad \underset{[0,T]}{\sup}   \norm{f_{\varepsilon}}_{\mathcal{H}_{2r}^{m-1}} < +\infty$. In order to define $\mathcal{N}_{m,2r}(T,f_{\varepsilon})$, we just have to prove that $\norm{\rho_{\varepsilon}}_{L^{2}\left([0,T],H_x^{m}\right)} < +\infty$. But because of the definition of weighted Sobolev norms, and with the use of the Cauchy-Schwarz inequality, we get that
\begin{equation*}
\norm{\rho_{\varepsilon}}_{H_x^{m}}\lesssim \norm{f_{\varepsilon}}_{\mathcal{H}_{2r}^{m}} < + \infty.
\end{equation*}

We have thus shown that the quantity $ \mathcal{N}_{m,2r}(T,f_{\varepsilon})$ is well defined and is continuous in $T$ for every $T<T^{*}$. This allows us to consider, for $R>0$ to be defined later,  

\begin{equation*}
T^{\varepsilon}=\sup \left\{T\in[0,T^{*}),\mathcal{N}_{m,2r}(T,f_{\varepsilon})\leq R\right\}.
\end{equation*}
By taking $R$ large enough, we have by continuity that $T^{\varepsilon}>0$. Of course $T^{\varepsilon}$ depends on $\varepsilon$. We want to show that by taking $R$ large enough (but independent of $\varepsilon$), $T^{\varepsilon}$ is uniformly bounded from below by some time $T>0$. Only the following two situations can happen\bigbreak
1. Either $T^{\varepsilon}=T^{*}$,\smallbreak
2. Or $T^{\varepsilon} < T^{*}$ and $\mathcal{N}_{m,2r}(T^{\varepsilon},f_{\varepsilon})=R$.

\bigbreak 
Let us analyze the first case. If $T^{\varepsilon}=T^{*}=+\infty$, then $\mathcal{N}_{m,2r}(T^{\varepsilon},f_{\varepsilon})\leq R$ for every $T>0$ and there is nothing to do, so we only have to consider when $T^{\varepsilon}=T^{*}<+\infty$. Actually, by energy estimates, we can show that this case is impossible. Indeed, we have the following proposition

\begin{prop}\label{PE}
Assume that $T^{\varepsilon}< + \infty$, then for every $f_{\varepsilon}$ solution of (\ref{SystP}), we have for some $C>0$ independent of $\varepsilon$ the estimate 
\begin{equation*}
\underset{[0,T^{\varepsilon})}{\sup}\norm{f_{\varepsilon}(t)}_{\mathcal{H}^{m}_{2r}}^{2}\leq \norm{f_{\varepsilon}^{0}}_{\mathcal{H}^{m}_{2r}}^{2}\exp\left[C\left(T^{\varepsilon}+\frac{1}{\varepsilon}(T^{\varepsilon})^{\frac{1}{2}}R\right)\right].
\end{equation*}
\end{prop}
\begin{proof}
Let $f_{\varepsilon}$ satisfying (\ref{SystP}). We thus have $\mathcal{T}f_{\varepsilon}=0$, and by using (\ref{ComLem}) 
\begin{equation*}
\mathcal{T}(\partial_{x}^{\alpha}\partial_{v}^{\beta}f_{\varepsilon}) =- \sum_{i=1}^{d}\mathds{1}_{\beta_{i}\ne 0}\partial_{x_{i}}\partial_{x}^{\alpha}\partial_{v}^{\overline{\beta}^{i}}f_{\varepsilon}-[\partial_{x}^{\alpha}\partial_{v}^{\beta},E_{\varepsilon}\cdot \nabla_{v}]f_{\varepsilon}.
\end{equation*}
Taking the scalar product with $(1+\vert v\vert ^{2})^{2r}\partial_{x}^{\alpha}\partial_{v}^{\beta}f_{\varepsilon}$ and summing for every $\vert\alpha\vert+\vert\beta\vert\leq m$, we obtain for the left hand size 
\begin{equation*}
\sum_{\vert\alpha\vert+\vert\beta\vert\leq m}\int_{\mathbb{T}^{d}}\int_{\mathbb{R}^{d}}\mathcal{T}(\partial_{x}^{\alpha}\partial_{v}^{\beta}f_{\varepsilon}) \cdot \partial_{x}^{\alpha}\partial_{v}^{\beta}f_{\varepsilon}=\frac{d}{dt}\norm{f_{\varepsilon}(t)}_{\mathcal{H}^{m}_{2r}}^{2}.
\end{equation*}
For the first term of the right hand size, we use the Cauchy-Schwarz inequality
\begin{equation*}
\left\vert\sum_{\vert\alpha\vert+\vert\beta\vert\leq m}\int_{\mathbb{T}^{d}}\int_{\mathbb{R}^{d}} \sum_{i=1}^{d}\chi\mathds{1}_{\beta_{i}\ne 0}\partial_{x_{i}}\partial_{x}^{\alpha}\partial_{v}^{\overline{\beta}^{i}}f_{\varepsilon}\cdot\chi \partial_{x}^{\alpha}\partial_{v}^{\beta}f_{\varepsilon}\right\vert\lesssim \norm{f_{\varepsilon}}_{\mathcal{H}^{m}_{2r}}^{2}.
\end{equation*}
For the second term, we use (\ref{Lemma}) with $s=m, \:\chi (v)=(1+\vert v\vert ^{2})^{r}$ and $s_{0}=m$, from which we deduce 
\begin{equation*}
\norm{\chi\left[\partial_{x}^{\alpha}\partial_{v}^{\beta},E_{\varepsilon}(x)\cdot \nabla_{v}\right]f_{\varepsilon}}_{L^{2}_{x,v}}\lesssim \norm{E_{\varepsilon}}_{H_x^{m}} \norm{f_{\varepsilon}}_{\mathcal{H}^{m}_{2r}},
\end{equation*}
and thus by using again Cauchy-Schwarz, we obtain
\begin{equation*}
\left\vert\int\chi\left[\partial_{x}^{\alpha}\partial_{v}^{\beta},E_{\varepsilon}(x)\cdot \nabla_{v}\right]f_{\varepsilon} \chi\partial_{x}^{\alpha}\partial_{v}^{\beta}f_{\varepsilon}\right\vert \lesssim \norm{E_{\varepsilon}}_{H_x^{m}} \norm{f_{\varepsilon}}_{\mathcal{H}^{m}_{2r}}^{2}.
\end{equation*}
We have, by elliptic regularity,
\begin{equation*}
\norm{E_{\varepsilon}}_{H_x^{m}}=\norm{\nabla_{x}(U*V_{\varepsilon})}_{H_x^{m}}\lesssim \frac{1}{\varepsilon} \norm{\rho_{\varepsilon}}_{H_x^{m}}.
\end{equation*}
Putting all together, we have shown that 
\begin{equation*}
\frac{d}{dt}\norm{f_{\varepsilon}(t)}_{\mathcal{H}^{m}_{2r}}^{2}\lesssim  \left(\frac{1}{\varepsilon} \norm{\rho_{\varepsilon}}_{H_x^{m}}+1\right) \norm{f_{\varepsilon}(t)}_{\mathcal{H}^{m}_{2r}}^{2}.
\end{equation*}
We integrate between $0$ and $t$ for $t\in[0,T^{\varepsilon})$. For some $C>0$ independent of $\varepsilon$, we get
\begin{equation*}
\norm{f_{\varepsilon}(t)}_{\mathcal{H}^{m}_{2r}}^{2}\leq \norm{f_{\varepsilon}^{0}}_{\mathcal{H}^{m}_{2r}}^{2}+C\int_{0}^{t}  \left(\frac{1}{\varepsilon} \norm{\rho_{\varepsilon}}_{H_x^{m}}+1\right) \norm{f_{\varepsilon}(s)}_{\mathcal{H}^{m}_{2r}}^{2},
\end{equation*} 
and we use the Gronwall inequality to show 
\begin{align*}
\underset{[0,T^{\varepsilon})}{\sup}{ \norm{f_{\varepsilon}(t)}}_{\mathcal{H}^{m}_{2r}}^{2}&\leq \norm{f_{\varepsilon}^{0}}_{\mathcal{H}^{m}_{2r}}^{2}+\exp\left[C\left(T^{\varepsilon}+\frac{1}{\varepsilon} (T^{\varepsilon})^{\frac{1}{2}}\norm{\rho_{\varepsilon}}_{L^{2}\left([0,T^{\varepsilon}\right),H_x^{m})}\right)\right]\\
&\leq\norm{f_{\varepsilon}^{0}}_{\mathcal{H}^{m}_{2r}}^{2}+\exp\left[C\left(T^{\varepsilon}+\frac{1}{\varepsilon} (T^{\varepsilon})^{\frac{1}{2}}\mathcal{N}_{m,2r}(T^{\varepsilon},f_{\varepsilon}) \right)\right].
\end{align*} 
Finally, since $\mathcal{N}_{m,2r}(T^{\varepsilon},f_{\varepsilon})\leq R$, we obtain the expected estimate.
\end{proof}

If we use the proposition for $T^{\varepsilon}=T^{*}$, we obtain that
\begin{equation*}
\underset{[0,T^{*})}{\sup}\norm{f_{\varepsilon}(t)}_{\mathcal{H}^{m}_{2r}}^{2}\leq \norm{f_{\varepsilon}^{0}}_{\mathcal{H}^{m}_{2r}}^{2}\exp\left[C\left(T^{*}+\frac{1}{\varepsilon}(T^{*})^{\frac{1}{2}}R\right)\right].
\end{equation*}
This means that the solution could be continued beyond $T^{*}$, and thus contradicts the definition of $T^{*}$, which shows that this case is impossible.
\bigbreak
We then have to consider the remaining case, $T^{\varepsilon} < T^{*}$ and $\mathcal{N}^{m}_{2r}(T^{\epsilon},f_{\varepsilon})=R$. Choosing $R$ large enough, the objective is to find some time $T^{\#}>0$ independent of $\varepsilon$, such that the equality 
\begin{equation*}
\mathcal{N}^{m}_{2r}(T,f_{\varepsilon})=R, 
\end{equation*}
cannot hold for any $T\in [0,T^{\#}]$, which will prove that $T^{\varepsilon}>T^{\#}>0$.
\bigbreak
We need to estimate $\mathcal{N}^{m}_{2r}(T,f_{\varepsilon})$ for $T<T^{\varepsilon}$. The easier part is the term $\norm{f_{\varepsilon}}_{\mathcal{H}^{m-1}_{2r}}$. We cannot use the previous estimates because they depend on $\varepsilon$ (we used the elliptic regularity provided by the Poisson equation). Nevertheless, we can still give estimates based on energy methods, this time being careful that every estimate must be independent of $\varepsilon$. In this proposition and in the following of this article, $\Lambda$ will stand for a generic continuous function, independent of $\varepsilon$, which is non-decreasing with respect to each of its arguments.

\bigbreak
\begin{prop}
For $m>2+d$ and $2r> d/2$, $f_{\varepsilon}$ solution of (\ref{SystP}) satisfies the estimate
\begin{equation}
\underset{[0,T]}{\sup}\norm{f_{\varepsilon}}_{\mathcal{H}^{m-1}_{2r}}\leq\norm{f_{\varepsilon}^{0}}_{\mathcal{H}^{m-1}_{2r}} \:+ T^{\frac{1}{2}}\Lambda (T,R),
\end{equation}
for all $T\in[0,T^{\varepsilon})$.
\end{prop}
\begin{proof}
Let $\alpha,\beta \in \mathbb{N}^{d}$ such that $\vert\alpha\vert +\vert \beta\vert = m-1$. We use (\ref{ComLem}) and, as before, we take the scalar product with  $\left(1+\vert v\vert^{2}\right)^{2r}\partial_{x}^{\alpha}\partial_{v}^{\beta}f_{\varepsilon}$, and we take the sum for every $\vert\alpha\vert +\vert \beta\vert = m-1$. Like we did previously, we use (\ref{Lemma}) with $s=m-1$, $\chi(v) =\left(1+\vert v\vert^{2}\right)^{r}$  and $s_{0}=m-1$. We obtain that
\begin{equation}\label{energy}
\frac{d}{dt}\norm{f_{\varepsilon}}_{\mathcal{H}^{m-1}_{2r}}^{2}\lesssim  \norm{f_{\varepsilon}}_{\mathcal{H}^{m-1+}_{2r}}^{2}+  \norm{f_{\varepsilon}}_{\mathcal{H}^{m-1}_{2r}}^{2} \norm{E_{\varepsilon}}_{H_x^{m-1}}.
\end{equation}
Integrating in time, we get that there exists a $C>0$ such that 
\begin{equation*}
\underset{[0,T]}{\sup}\norm{f_{\varepsilon}}_{\mathcal{H}^{m-1}_{2r}}\leq\norm{f_{\varepsilon}^{0}}_{\mathcal{H}^{m-1}_{2r}}+C\underset{[0,T]}{\sup}\norm{f_{\varepsilon}}_{\mathcal{H}^{m-1}_{2r}}\left(T+\int_{0}^{T} \norm{E_{\varepsilon}}_{H_x^{m-1}} dt \right).
\end{equation*}
We can now use the following estimate, independent of $\varepsilon$ 
\begin{equation}\label{Optimal}
\norm{E_{\varepsilon}}_{H_x^{m-1}}=\norm{\nabla_{x}(U*V_{\varepsilon})}_{H_x^{m-1}}\lesssim \norm{\rho_{\varepsilon}}_{H_x^{m}}.
\end{equation}
And we still have $\mathcal{N}_{m,2r}(T,f_{\varepsilon})\leq R$, (because $T\in[0,T^{\epsilon}])$. We have thus shown that  
\begin{equation*}
\underset{[0,T]}{\sup}\norm{f_{\varepsilon}}_{\mathcal{H}^{m-1}_{2r}}\leq\norm{f_{\varepsilon}^{0}}_{\mathcal{H}^{m-1}_{2r}} + CR(T+T^{\frac{1}{2}}R),
\end{equation*}
which concludes the proof of the lemma.

\end{proof}
\begin{Rem}
As said in the introduction, our estimates could be sharper. This can been seen in (\ref{Optimal}), which is far from being optimal, because we do not use the regularity provided by the convolution with $U$. Nevertheless, this regularization is not needed to prove Theorem \ref{theorem2}, thus we chose not to exploit it for simplicity.
\end{Rem}

\section{Estimates of the density term}

\subsection{Introduction of the $f_{I,J}$} 

Now we need to tackle the second term of $\mathcal{N}_{m,2r}(T,f_{\varepsilon})$, which is the norm of the density $\norm{\rho_{\varepsilon}}_{L^{2}\left([0,T],H_x^{m}\right)}$. We can try to apply the operator $\partial_x^{\alpha}$ to (\ref{SystP}) with $\vert\alpha\vert=m$ , but this involves commutator terms such as $\partial_x^{\alpha '} E_{\varepsilon}\cdot \nabla_{v}\partial_x^{\alpha-\alpha '}f_{\varepsilon}$ which contain $m$ order derivatives of $f_{\varepsilon}$ when $\vert\alpha ' \vert =1$, and those cannot be estimated uniformly  in $\varepsilon$. To get rid of this problem, we choose to apply a larger class of differential operators to (\ref{SystP}).

\begin{definition}
For $I=(i_{1},...,i_{d}), J=(j_{1},...,j_{d})\in \mathbb{N}^{d}, \vert I\vert + \vert J\vert =m$, we define 

\begin{equation*}
f_{I,J}:=\partial_{x}^{I}\partial_{v}^{J}f_{\varepsilon}.
\end{equation*}

\end{definition} 

Note that the $(f_{I,J})$ contain all the $\partial^{I}_x f_{\varepsilon}$ with $\vert I\vert =m$. By applying $\partial_{x}^{I}\partial_{v}^{J}$ to (\ref{SystP}), we find that the $(f_{I,J})$ satisfy a differential system, which is the purpose of the following lemma. 
\begin{lemma}\label{RReste}
We assume that $m>3+d$ and $2r>d$. For all $T<T^{\varepsilon}$, and for all $I,J\in\mathbb{N}^{d}$, we have, for $f_{\varepsilon}$ satisfying (\ref{SystP}), that $f_{I,J}$ is solution of
\begin{equation}\label{IJ}
\mathcal{T}(f_{I,J})+\partial_{x}^{I}\partial_{v}^{J}E_{\varepsilon} \cdot \nabla_{v}f_{\varepsilon} +\mathcal{M} _{I,J} \mathcal{F}=\mathcal{V}_{I,J},
\end{equation} 
where 
\begin{equation*}
 \mathcal{F}=(f_{I,J})_{I,J\in\mathbb{N}^{d},\vert I\vert + \vert J\vert =m},\quad \mathcal{M} _{I,J} \mathcal{F}=\sum_{k=1}^{d}\mathds{1}_{i_{k}\ne 0}f_{\hat{I}^{k},\overline{J}^{k}}+\sum_{p=1}^{d}\sum_{k=1}^{d}\mathds{1}_{i_p \ne 0}\partial_{x_p}E_{\varepsilon} f_{\overline{I}^{p},\hat{J^{k}}},
\end{equation*}
with 
\begin{equation*}
\hat{I}^{k}=(i_{1}...i_{k-1},i_{k}+1,i_{k+1}...i_{d}),\quad \overline{I}^{k}=(i_{1}...i_{k-1},i_{k}-1,i_{k+1}...i_{d}),
\end{equation*}
and $\mathcal{V}=(\mathcal{V}_{I,J})_{I,J\in\mathbb{N}^{d}}$ is a remainder, which means that for every $T<T^{\varepsilon}$
\begin{equation*}
\norm{\mathcal{V}}_{L^{2}([0,T],\mathcal{H}^{0}_{r})}\leq \Lambda(T,R).
\end{equation*}
\end{lemma}
We therefore obtain that $ \mathcal{F}=(f_{I,J})$ satisfies a system which is coupled through the linear term $\mathcal{M} _{I,J} \mathcal{F}$. We now prove this lemma.

\begin{proof}
We know that $f_{\varepsilon}$ solves (\ref{SystP}), so we have 

\begin{equation*}
\partial_{x}^{I}\partial_{v}^{J}(\mathcal{T}(f_{\varepsilon})) =0.
\end{equation*}
By using (\ref{ComLem}), we obtain

\begin{equation*}
\mathcal{T}(f_{I,J})+\partial_{x}^{I}\partial_{v}^{J}E_{\varepsilon} \cdot \nabla_{v}f_{\varepsilon} +\mathcal{M} _{I,J} \mathcal{F}=\mathcal{V}_{I,J},
\end{equation*}
with
\begin{equation*}
\mathcal{V}_{I,J}= \sum_{k=2}^{\min(\vert I\vert ,m-1)}\sum_{\sigma,\vert\sigma\vert=k}C_{I,J,k,\sigma}\:\partial^{\sigma}_x E_{\varepsilon}\:\cdot\nabla_x \partial^{I-\sigma}_x \partial^{J}_v f_{\varepsilon} . 
\end{equation*}
Notice that because $\vert\sigma\vert=k\geq 2$, we have in particular that $1+\vert I\vert -\vert\sigma\vert +\vert J\vert\leq m-1$. Also, in the case $\vert I\vert \leq 1$, we simply have $\mathcal{V}_{I,J}=0$.\\
We want to estimate $\norm{\mathcal{V}_{I,J}}_{L^{2}([0,T],\mathcal{H}^{0}_{r})}$. We have that 
\begin{equation*}
\norm{\mathcal{V}_{I,J}}_{\mathcal{H}^{0}_{r}}\lesssim \norm{E_{\varepsilon}}_{H_x^{m-1}}\norm{f_{\varepsilon}}_{\mathcal{H}^{m-1}_{r}}.
\end{equation*}
And, because $\norm{E_{\varepsilon}}_{H_x^{m-1}}\lesssim \norm{\rho_{\varepsilon}}_{H_x^{m}}$ (estimate independent of $\varepsilon$),
\begin{equation*}
\norm{\mathcal{V}_{I,J}}_{L^{2}([0,T],\mathcal{H}^{0}_{r})}\lesssim  \norm{\rho_{\varepsilon}}_{L^{2}([0,T],H_x^{m})}\norm{f_{\varepsilon}}_{L^{\infty}([0,T],\mathcal{H}^{m-1}_{r})}, 
\end{equation*}
and we obtain that
\begin{equation*}
\norm{\mathcal{V}}_{L^{2}([0,T],\mathcal{H}^{0}_{r})}\leq \Lambda(T,R),
\end{equation*}
which ends the proof.
\end{proof}

\subsection{Straightening the transport vector field}

In the following, we make a change of variable to straighten the vector field 
\begin{equation*}
\partial_t +v\cdot\nabla_x+E\cdot\nabla_v\quad \text{into}\quad \partial_t+\Phi(t,x,v)\cdot \nabla_x,
\end{equation*}
where $\Phi$ is defined in the following lemma.
\begin{lemma}
Let $f_{I,J}$ a solution of (\ref{IJ}). We consider $\Phi(t,x,v)$ a smooth solution of the Burgers equation

\begin{equation}\label{Burgers}
\partial_{t}\Phi +\Phi \cdot \nabla_{x} \Phi =E_{\varepsilon},
\end{equation}
such that the Jacobian matrix $(\nabla_{v}\Phi)$ is inversible. We define $g_{I,J}$ by

\begin{equation*}
g_{I,J}(t,x,v):= f_{I,J} (t,x,\Phi).
\end{equation*}
Then $g_{I,J}$ is solution of the equation

\begin{equation}\label{NewEq}
\partial_{t}g_{I,J} +\Phi\cdot \nabla_{x} g_{I,J} +\partial_{x}^{I}\partial_v^J E_{\varepsilon}\cdot \left(\nabla_{v}f_{\varepsilon}\right)(t,x,\Phi)+\mathcal{M}_{I,J} \mathcal{G} =\mathcal{V}_{I,J}(t,x,\Phi),
\end{equation}
where $\mathcal{G}=(g_{I,J})_{I,J\in\mathbb{N}^{d},\vert I\vert + \vert J\vert =m}$.
\end{lemma}
\begin{proof}
The proof is based on a simple calculation 
\begin{equation*}
\begin{split}
\partial_{t}g_{I,J} +\Phi\cdot \nabla_{x} g_{I,J} +\partial_{x}^{I}\partial_v^J E_{\varepsilon}\cdot \left(\nabla_{v}f_{\varepsilon}\right)(t,x,\Phi)+\mathcal{M}_{I,J} \mathcal{G} =\mathcal{V}_{I,J}(t,x,\Phi)\\+ ^{t}(\nabla_{v}\Phi)^{-1} \nabla_{v}g_{I,J}\cdot \left(\partial_{t}\Phi+\Phi\cdot\nabla_{x}\Phi-E_{\varepsilon}\right),
\end{split}
\end{equation*}
hence the result when (\ref{Burgers}) is satisfied. 
\end{proof}

\begin{Rem}
Writing $\mathcal{J}(t,x,v)= \vert \det\nabla_{v}\Phi(t,x,v)\vert$, notice that 
\begin{equation}\label{notice}
\int_{\mathbb{R}^{d}}g_{I,J}\mathcal{J} dv=\int_{\mathbb{R}^{d}}f_{I,J} dv.
\end{equation}
\end{Rem}
Because we introduced the function $\Phi$ in our equations, we shall estimate its Sobolev norms, which is the purpose of the following lemma.

\begin{lemma}\label{estphi}
Suppose that $m>3+d$, there exists $T_{0}$ (depending on $R$ but independent on $\varepsilon$) such that for all $T<\min (T_{0},T^{\varepsilon})$, there exists a unique smooth solution on $[0,T]$ of the Burgers equation (\ref{Burgers}) with initial condition $\Phi\vert_{t=0}=v$. \newline
Moreover, for all $T<\min (T_{0},T^{\varepsilon})$, we have the following estimates: 
\begin{equation}
\underset{[0,T]}{\sup}\norm{\Phi-v}_{W^{k,\infty}_{x,v}}+\underset{[0,T]}{\sup}\norm{\frac{1}{\left(1+\vert v\vert ^{2}\right)^{\frac{1}{2}}}\partial_{t}\Phi}_{W^{k-1,\infty}_{x,v}}\leq T^{\frac{1}{2}}\Lambda(T,R),\quad k<m-d/2-1,
\end{equation}
for $\vert\alpha\vert\leq m-1$ and $\vert\beta\vert \leq m-2$
\begin{equation}
\underset{[0,T]}{\sup}\,\underset{v}{\sup}\norm{\partial^{\alpha}_{x,v}\left(\Phi-v\right)}_{L^{2}_{x}}+\underset{[0,T]}{\sup}\,\underset{v}{\sup}\norm{\frac{1}{\left(1+\vert v\vert ^{2}\right)^{\frac{1}{2}}}\partial^{\beta}_{x,v}\partial_{t}\Phi}_{L^{2}_{x}}\leq T^{\frac{1}{2}}\Lambda(T,R).
\end{equation}
\end{lemma}
A very similar lemma (Lemma 11) has been proved in \cite{HKR}, we refer to it for a complete proof.
\bigbreak

Now we have to consider our new equation (\ref{NewEq}), that we rewrite as
\begin{equation}\label{eqG}
\partial_{t}g_{I,J}+\Phi \cdot\nabla_{x} g_{I,J} +\partial^{I}_{x}\partial^{J}_{v}E_{\varepsilon}\cdot(\nabla_{v}f_{\varepsilon})(t,x,\Phi)+\mathcal{M}_{I,J}\mathcal{G} =S_{I,J},
\end{equation}
where $S_{I,J}(t,x,v)= \mathcal{V}_{I,J} (t,x,\Phi(t,x,v))$.\\
The next step is to introduce the flow of the equation that we denote by $X(t,s,x,v), 0\leq s,t\leq T$ and is given as the solution of 
\begin{equation*}
\partial_{t}X(t,s,x,v) =\Phi (t,X(t,s,x,v),v), \quad X(s,s,x,v)=x.
\end{equation*}
We have to control the Sobolev norms of $X$.
\begin{lemma}\label{LemmaPsi}
For all $t,s,\: 0\leq s\leq t\leq T$ and $m>3+d$, we write  
\begin{equation*}
X(t,s,x,v)=x+(t-s)\left(v+\tilde{X}(t,s,x,v)\right).
\end{equation*}

We have that $\tilde{X}$ satisfies, for $\vert\alpha\vert <m-d/2-1$, $\vert\beta\vert <m-d/2-2$
\begin{equation}
\underset{t,s\in[0,T]}{\sup}\norm{\partial^{\alpha}_{x,v}\tilde{X}(t,s,x,v)}_{L^{\infty}_{x,v}}+\underset{t,s\in[0,T]}{\sup}\norm{\dfrac{1}{(1+\vert v\vert^{2})^{\frac{1}{2}}}\partial^{\beta}_{x,v}\partial_{t}\tilde{X}(t,s,x,v)}_{L^{\infty}_{x,v}}\leq T^{\frac{1}{2}}\Lambda(T,R).
\end{equation}

Furthermore, there exists $\hat{T}_{0}(R)>0$ small enough such that for $T\leq \min(T_{0},\hat{T}_{0}, T^{\varepsilon})$, we get that $x\mapsto x+(t-s) \tilde{X}(t,s,x,v)$ is a diffeomorphism and that, for  $\vert\alpha\vert <m-1$, $\vert\beta\vert <m-2$,

\begin{equation}
\underset{t,s\in[0,T]}{\sup}\underset{v}{\sup}\norm{\partial^{\alpha}_{x,v}\tilde{X}(t,s,x,v)}_{L^{2}_x}+\underset{t,s\in[0,T]}{\sup}\underset{v}{\sup}\norm{\dfrac{1}{(1+\vert v\vert^{2})^{\frac{1}{2}}}\partial^{\beta}_{x,v}\partial_{t}\tilde{X}(t,s,x,v)}_{L^{2}_x}\leq T^{\frac{1}{2}}\Lambda(T,R).
\end{equation}

Finally, there exists $\Psi(t,s,x,v)$ such that for $t,s\in[0,T]$ and $T\leq \min(T_{0},\hat{T}_{0}, T^{\varepsilon})$, we have,
\begin{equation}
X(t,s,x,\Psi(t,s,x,v))=x+(t-s)v,
\end{equation}
with $\Psi$ satisfying the estimates, for $\vert\alpha\vert <m-d/2-1$, $\vert\beta\vert <m-d/2-2$  

\begin{equation}
\underset{t,s\in[0,T]}{\sup}\norm{\partial^{\alpha}_{x,v}(\Psi(t,s,x,v)-v)}_{L^{\infty}_{x,v}}+\underset{t,s\in[0,T]}{\sup}\norm{\dfrac{1}{(1+\vert v\vert^{2})^{\frac{1}{2}}}\partial^{\beta}_{x,v}\partial_{t}\Psi(t,s,x,v)}_{L^{\infty}_{x,v}}\leq T^{\frac{1}{2}}\Lambda(T,R),
\end{equation}
for $\vert\alpha\vert <m-1$, $\vert\beta\vert <m-2$

\begin{equation}
\underset{t,s\in[0,T]}{\sup}\underset{v}{\sup}\norm{\partial^{\alpha}_{x,v}(\Psi(t,s,x,v)-v)}_{L^{2}_x}+\underset{t,s\in[0,T]}{\sup}\underset{v}{\sup}\norm{\dfrac{1}{(1+\vert v\vert^{2})^{\frac{1}{2}}}\partial^{\beta}_{x,v}\partial_{t}\Psi(t,s,x,v)}_{L^{2}_x}\leq T^{\frac{1}{2}}\Lambda(T,R).
\end{equation}

\end{lemma}
Once again, a similar lemma (Lemma 13) has already been proved in \cite{HKR}.
\bigbreak 
In order to control the linear part, we define the tensor $\mathcal{M}$ by $(\mathcal{M}H)_{I,J}=M_{I,J}H$ and for $0\leq s,t\leq T$, $x\in \mathbb{T}^{d}$, $v\in \mathbb{R}^{d}$, and we introduce $\mathfrak{M}(t,s,x,v)$ as the solution of 

\begin{equation*}
\partial_{t}\mathfrak{M}(t,s,x,v) =\mathcal{M}(t,x,\Phi(t,x,v))\mathfrak{M}(t,s,x,v),\quad \mathfrak{M}(s,s,x,v) =I ,
\end{equation*} 
whose existence and uniqueness is guaranteed by the Cauchy-Lipschitz theorem. \\
By a Gronwall type argument and thanks to Lemma \ref{estphi}, we can show that for $k<m-2$
\begin{equation}
\underset{0\leq s,t\leq T}{\sup}\left(\norm{\mathfrak{M}}_{W^{k,\infty}_{x,v}}+\norm{\partial_{t}\mathfrak{M}}_{W^{k,\infty}_{x,v}}+\norm{\partial_{s}\mathfrak{M}}_{W^{k,\infty}_{x,v}}\right)\leq \Lambda(T,R).
\end{equation}

\subsection{Introduction of the average operator}
We define in the next lemma the fundamental operator $K_G$ which was introduced in \cite{HKR}
\begin{lemma}\label{AOP}
For a smooth function $G(t,s,x,v)$, we define the integral operator $K_{G}$ acting on $F(t,x)$ by
\begin{equation}\label{KG}
K_{G}(F)(t,x) =\int_{0}^{t}\int_{\mathbb{R}^d} (\nabla_{x}F)(s,x-(t-s)v)\cdot G(t,s,x,v) dvds.
\end{equation}
For $f_{\varepsilon}$ satisfying (\ref{SystP}) and $\rho_{\varepsilon}=\int f_{\varepsilon} dv$, the functions $\partial^{I}_{x}\rho_{\varepsilon}$  with $\vert I\vert =m$ satisfy the equation: 
\begin{equation*}
\partial^{I}_{x}\rho_{\varepsilon} =\sum_{K\in\{1,...,d\}^{m}} K_{H_{(K,0),(I,0)}}\left(U*\left(\left(I-\varepsilon^{2}\Delta\right)^{-1}\partial_{x}^{K}\rho_{\varepsilon}\right)\right)+\mathcal{R}_{I,0},
\end{equation*}
with
\begin{equation*}
\begin{split}
&H_{(K,L),(I,J)}=\\&\mathfrak{M}_{(K,L),(I,J)}(s,t,x,\Psi(s,t,x,v))(\nabla_{v}f)(s,x-(t-s)v,\Psi (s,t,x,v))\mathcal{J}(t,s,\Psi(s,t,x,v))\tilde{\mathcal{J}}(s,t,x,v),
\end{split}
\end{equation*}
\begin{equation*}
\mathcal{J}=\vert det\nabla_{v}\Phi(t,x,\Psi(s,t,x,v))\vert, \quad \tilde{\mathcal{J}}=\vert det\nabla_{v} \Psi(s,t,x,v) \vert,
\end{equation*}
and $\mathcal{R}_{I,0}$ satisfies for $T$ small enough, the estimate
\begin{equation*}
\norm{\mathcal{R}_{I,0}}_{L^{2}([0,T],L^{2}_{x})} \lesssim T^{\frac{1}{2}} \Lambda (T,R).
\end{equation*} 
\end{lemma}
\begin{proof}
We first introduce the notations
\begin{equation*}
\eta(t,x,v) = \left( \partial^I_x\partial ^J_v E_{\varepsilon}(t,x) \cdot \nabla_{v}f_{\varepsilon}(s,x,\Phi(t,x,v))   \right)_{I,J}, \quad S=(S_{I,J})_{I,J}.
\end{equation*}

This allows us to put the system (\ref{eqG}) under the following equation satisfied by $\mathcal{G}$
\begin{equation*}
\partial_{t}\mathcal{G}(t,x,v)+\Phi \cdot\nabla_{x}\mathcal{G}(t,x,v) +\eta(t,x,v) +\mathcal{M}\mathcal{G}(t,x,v) =S(t,x,v).
\end{equation*}

Integrating by respect with the time variable the expression  $\partial_{t}(\mathfrak{M}(t,s,x,v)\mathcal{G}(t,X(t,s,x,v),v))$ we obtain that
\begin{equation*}
\begin{split}
\mathcal{G}(t,x,v)=\mathfrak{M}(0,t,x,v) \mathcal{G}^{0}(X(0,t,x,v),v) +\int_{0}^{t}\mathfrak{M}(s,t,x,v){S} (s,X(s,t,x,v),v) ds &\\-\int_{0}^{t} \mathfrak{M}(s,t,x,v)\eta(s,X(s,t,x,v),v)ds,
\end{split}
\end{equation*}
with $\mathcal{G}^{0}=(g^{0}_{I,J})_{I,J}$. We multiply the equation by $\mathcal{J}$ and then with integrate by respect with the variable $v$, which yields 
\begin{equation}\label{MajG}
\int_{\mathbb{R}^{d}} \mathcal{G}(t,x,v) \mathcal{J} (t,x,v) dv =\mathcal{I}_{0} +\mathcal{I}_{F}-\int_{\mathbb{R}^{d}}\int_{0}^{t}\mathfrak{M}(s,t,x,v) \eta(s,X(s,t,x,v),v)\mathcal{J}(t,x,v)dvds,
\end{equation}
with 
\begin{equation*}
\mathcal{I}_{0}=\int_{\mathbb{R}^{d}}\mathfrak{M}(0,t,x,v) \mathcal{G}^{0}(X(0,t,x,v),v) \mathcal{J}(t,x,v) dv,
\end{equation*}
\begin{equation*}
\mathcal{I}_{F} =\int_{0}^{t}\int_{\mathbb{R}^{d}}\mathfrak{M}(s,t,x,v) S (s,X(s,t,x,v),v)\mathcal{J}(t,x,v)dsdv.
\end{equation*}
We want to show that $\mathcal{I}_{0},\mathcal{I}_{F}$ can be considered as remainders. Let us recall some of the previous estimates on $\Phi$ and $\mathfrak{M}$ 
\begin{equation}\label{EstM}
\underset{0\leq s,t\leq T}{\sup}\norm{\mathfrak{M}(t,s)}_{L^{\infty}_{x,v}}\leq \Lambda(T,R).
\end{equation}
\begin{equation*}
\underset{[0,T]}{\sup}\norm{\Phi(t)-v}_{W^{1,\infty}_{x,v}}\leq \Lambda(T,R).
\end{equation*}
Now we can give the following estimate 
\begin{equation*}
\begin{split}
\left\vert \int \mathfrak{M}(t,0,x,v)\mathcal{G}^{0} (X(0,t,x,v),v)\mathcal{J}(t,x,v)dv\right\vert&\\ \leq \underset{0\leq s,t\leq T}{\sup}\norm{\mathfrak{M}(s,t)}_{L^{\infty}_{x,v}} \underset{0\leq t\leq T}{\sup}\norm{\mathcal{J}(t)}_{L^{\infty}_{x,v}} \int\vert\mathcal{G}^{0} (X(0,t,x,v),v)\vert dv&\\ \leq \Lambda(T,R) \sum_{I,J} \int \vert g^{0}_{I,J} (X(0,t,x,v),v)\vert dv.
\end{split} 
\end{equation*} 
Thus, we have  
\begin{equation*}
\norm{\mathcal{I}_{0}}_{L^{2}([0,T],L^{2}_{x})}\leq \Lambda (T,R) \sum_{I,J} \norm{\int_{v}\norm{g^{0}_{I,J} (X(0,t,\cdot,v),v)}_{L^{2}_{x}}dv}_{L^{2}[0,T]}.
\end{equation*}
To estimate the last term, we use the change of variable in $x$,  $y=X (0,t,x,v)+tv=x-t\tilde{X}(0,t,x,v)$ and, thanks to the estimates on $\tilde{X}$ 
\begin{equation*}
\underset{t,s\in[0,T]}{\sup}\norm{\partial_{x}\tilde{X}(t,s,x,v)}_{L^{\infty}_{x,v}}\leq \Lambda(T,R),
\end{equation*}
we get that 
\begin{equation*}
\norm{g^{0}_{I,J} (X(0,t,\cdot,v),v)}_{L^{2}_{x}}dv\leq \Lambda (T,R) \norm{g^{0}_{I,J} (\cdot-tv),v}_{L^{2}_x} \leq \Lambda (T,R) \norm{g^{0}_{I,J}(\cdot,v) }_{L^{2}_x}.
\end{equation*}
Then by Cauchy-Schwarz, we deduce 
\begin{equation*}
\norm{\mathcal{I}_{0}}_{L^{2}([0,T],L^{2}_{x})}\leq T^{\frac{1}{2}} \Lambda (T,R) \left( \int_{\mathbb{R}^{d}} \frac{dv}{(1+\vert v\vert^{2})^{r}}\right) ^{\frac{1}{2}}\sum_{I,J} \norm {g^{0}_{I,J} }_{\mathcal{H}^{0}_{r}}.
\end{equation*}
Again, we can make the following change of variable
\begin{equation*}
\int_{x,v} \vert g^{0}_{I,J}(x,v)\vert ^2 dxdv= \int_{x,v} \vert f^{0}_{I,J}(x,v)\vert ^2 \mathcal{J}(0,x,v) dxdv \leq \Lambda (T,R) \norm{f^0_{I,J}}_{\mathcal{H}^0_{r}}^2.
\end{equation*}
Finally, we can use the fact that $\norm{f^0_{I,J}}_{\mathcal{H}^0_{r}}\leq \norm{f_{\varepsilon}}_{\mathcal{H}^m_{r}}\leq R$ to conclude that
\begin{equation*}
\norm{\mathcal{I}_{0}}_{L^{2}([0,T],L^{2}_{x})}\leq T^{\frac{1}{2}} \Lambda (T,R).
\end{equation*}
Let us show with similar arguments that 
\begin{equation*}
\norm{\mathcal{I}_{f}}_{L^{2}([0,T],L^{2}_{x})}\leq T \Lambda (T,R).
\end{equation*}
First, we use once again (\ref{EstM}) to show that
\begin{equation*}
\norm{\mathcal{I}_{F}}_{L^{2}([0,T],L^{2}_{x})}\leq \Lambda (T,R) \sum_{I,J} \norm{\int_{0}^{t}\int_{v}\norm{S_{I,J} (s,X(s,t,\cdot,v),v)}_{L^{2}_{x}}dvds}_{L^{2}[0,T]}.
\end{equation*}
As we did just above, we use the change of variable in $x$,  $y=X (s,t,x,v)+(t-s)v=x-(t-s)\tilde{X}(s,t,x,v)$, and the estimates on $\tilde{X}$ to show that 
\begin{equation*}
\norm{S_{I,J} (s,X(s,t,\cdot,v),v)}_{L^{2}_{x}}\leq \Lambda (T,R) \norm{S_{I,J} (s,\cdot,v)}_{L^{2}_{x}}.
\end{equation*}
Integrating in $v$, we get by the Cauchy-Schwarz inequality
\begin{equation*}
\int_{v}\norm{S_{I,J} (s,X(s,t,\cdot,v),v)}_{L^{2}_{x}}dv\leq \Lambda (T,R) \norm{S_{I,J} (s)}_{\mathcal{H}^{0}_{r}}.
\end{equation*}
So we finally arrive at 
\begin{equation*}
\norm{\mathcal{I}_{F}}_{L^{2}([0,T],L^{2}_{x})}\leq\Lambda (T,R) \norm{\int_{0}^{t} \norm{S_{I,J} (s)}_{\mathcal{H}^{0}_{r}}}_{L^{2}[0,T]}\leq \Lambda(T,R) T \norm{S}_{L^2 ([0,T],\mathcal{H}^{0}_{r})}.
\end{equation*}
Recalling that $S(t,x,v)= \mathcal{V}(t,x,\Phi(t,x,v))$, we can use one last change of variable, and the estimates on the derivatives of $\Phi$, to show that 
\begin{equation*}
\norm{S}_{L^2 ([0,T],\mathcal{H}^{0}_{r})}\leq \Lambda(T,R) \norm{\mathcal{V}}_{L^2([0,T],\mathcal{H}^{0}_{r})}\leq \Lambda (T,R),
\end{equation*}
where the last inequality comes from the fact that $\mathcal{V}$ is a remainder term, which we proved in lemma \ref{RReste}.\\
Thus, we have proved that
\begin{equation*}
\norm{\mathcal{I}_{F}}_{L^{2}([0,T],L^{2}_{x})}\leq T \Lambda (T,R).
\end{equation*}
Let us go back to (\ref{MajG}). Thanks to the results on $\mathcal{I}_0$ and $\mathcal{I}_F$, and using (\ref{notice}), we have
\begin{equation*}
\begin{split}
&\partial^{I}_{x}\rho_{\varepsilon} =\mathcal{R}_{I,0} - \\&\int_{0}^{t} \int \sum_{K} \mathfrak{M}_{(K,0),(I,J)}(\partial^{K}_{x}E) (s,X(s,t,x,v)) \cdot (\nabla_{v}f)(s,X(s,t,x,v),\Phi(s,X(s,t,x,v),v))\mathcal{J}(t,x,v)dv ds, 
\end{split}
\end{equation*} 
with 
\begin{equation*}
\norm{\mathcal{R}_{I,0}}_{L^{2}([0,T],L^{2}_{x})}\lesssim T^{\frac{1}{2}} \Lambda (T,R). 
\end{equation*}
We can finally use the change of variable $v=\Psi (s,t,x,w)$ provided by lemma \ref{LemmaPsi} to obtain that
\begin{equation*}
\partial^{I}_{x}\rho_{\varepsilon} = - \int_{0}^{t} \int \sum_{K} (\partial^{K}_{x}E)(s,x-(t-s)v)\cdot H_{(K,0),(I,0)}(t,s,x,v)dvds + \mathcal{R}_{I,0},
\end{equation*}
with 
\begin{equation*}
\begin{split}
&H_{(K,L),(I,J)}=\\&\mathfrak{M}_{(K,L),(I,J)}(s,t,x,\Psi(s,t,x,v))(\nabla_{v}f)(s,x-(t-s)v,\Psi (s,t,x,v))\mathcal{J}(t,s,\Psi(s,t,x,v))\tilde{\mathcal{J}}(s,t,x,v),
\end{split}
\end{equation*}
which gives the result.
\end{proof}

\subsection{Focus on the operator $K_{G}$}
In order to control the norms of the derivatives of $\rho_{\varepsilon}$, we have to understand better the operator $K$. Following \cite{HKR}, let us first introduce a new norm.

\begin{definition}
For $T>0$, we define
\begin{equation*}
\norm{G}_{T,s_1,s_2}= \underset {0\leq t\leq T}{\sup}\left( \sum_{k} \underset {0\leq s\leq T}{\sup}\underset {\xi}{\sup} \left((1+\vert k\vert)^{s_2}(1+\vert \xi\vert) ^{s_1}    \left\vert   (\mathcal{F}_{x,v}G)(t,s,k,\xi)\right\vert\right)^{2}     \right)^{\frac{1}{2}}.
\end{equation*}
\end{definition}

\begin{prop}\label{MdK}
There exists $C>0$ such that for every $T>0$,  $\alpha\in (0,1)$,  for every $G$ satisfying $\norm{G}_{T,s_{1},s_{2}} <\infty$ and for all $s_{1}>1$, $s_{2}>d/2$, we have  
\begin{equation*}
\norm{K_{G}(F)}_{L^{2}([0,T],L^{2}_{x})}\leq C \,T^{\alpha/2} \,\norm{G}_{T,s_{1},s_{2}} \norm{F}_{L^{2}([0,T],H^{\alpha}_{x})},\quad  \forall F\in L^{2}([0,T],H^{\alpha}_{x}).
\end{equation*}
\end{prop}
For practical uses, it is convenient to relate the norm $\norm{G}_{T,s_{1},s_{2}} <\infty$ to a more tractable norm. From \cite{HKR}, we know that if $p>1+d$, $\sigma > d/2$, we can find $s_{2}>d/2$ et $s_{1}>1$ such that 
\begin{equation*}
\norm{G}_{T,s_{1},s_{2}}\leq \underset{0\leq s,t\leq T}{\sup}\norm{G(t,s)}_{\mathcal{H}^{p}_{\sigma}}.
\end{equation*}
In the expression of the operator $K_G$ (\ref{KG}), there seems to have of loss of a derivative in $x$ , but this proposition shows that the operator is actually continuous from $L^{2}([0,T],H^{\alpha}_{x})$ into $L^{2}([0,T],L^{2}_{x})$ provided that $G$ is smooth enough. Once again, we emphasize that this is a key propriety in the proof of Theorem $\ref{theorem2}$. We are able to gain regularity in $x$ by integrating over $v$ and $t$. This propriety explains why we can estimate the $H^{m}$ norm of $\rho_{\varepsilon}$ without loss of derivative.      
\bigbreak 
Proposition \ref{MdK} remains true for $\alpha=0$, this was proved in \cite{HKR}.

\begin{proof}

By using Fourier series in $x$, we write that 
\begin{equation*}
F(t,x)=\sum_{k\in\mathbb{Z}^d}\hat{F}_{k}(t)e^{ik\cdot x}.
\end{equation*}
By definition of $K_{G}$, we have

\begin{align*}
K_{G}F(t,x)&= \int_{0}^{t}\sum_{k} \hat{F}_{k} (s) e^{i k\cdot x}\cdot \int e^{-i k\cdot v(t-s)}G(t,s,x,v) dv ds \\
&=  \int_{0}^{t} \sum_{k} \hat{F}_{k}(s) e^{i k\cdot x} ik \cdot (\mathcal{F}_{v}G) (t,s,x,k(t-s))ds,
\end{align*}
where $\mathcal{F}_{v}G$ is the Fourier transform of $G(t,s,x,v)$ with respect to the last variable. By Fourier expanding in the $x$ variable, we deduce that
\begin{equation*}
K_{G}F(t,x) = \sum_{k} e^{ik\cdot x} \sum _{l} e^{ik\cdot x} \int_{0}^{t} \hat{F}_{k} (s) ik \cdot (\mathcal{F}_{x,v}G)(t,s,l,k(t-s)) ds.  
\end{equation*}
Changing $l$ into $l+k$ we can rewrite this expression as 

\begin{equation*}
K_{G}F(t,x)=\sum_{l} e^{ik\cdot x} \left(\sum_{k} \int_{0}^{t} \hat{F}_{k} (s) ik \cdot (\mathcal{F}_{x,v}G)(t,s,l-k,k(t-s))ds\right).
\end{equation*}
From the Bessel-Parseval identity, we infer that 

\begin{equation*}
\norm{K_{G}}^{2}_{L^{2}_{x}}= \sum_{l} \left\vert\sum_{k} \int_{0}^{t} \hat{F}_{k}(s) ik \cdot (\mathcal{F}_{x,v}G)(t,s,l-k,k(t-s))ds\right\vert ^{2}.
\end{equation*}
By using Cauchy-Schwarz for $t$ and $k$, we have 
\begin{equation*}
\begin{split}
\norm{K_{G}}^{2}_{L^{2}_{x}}\lesssim \sum_{l}  (\sum_{k} \int_{0}^{t} \vert\hat{F}_{k}(s) &\vert ^{2} \vert k \cdot (\mathcal{F}_{x,v} G)(t,s,l-k,k(t-s)) \vert ds\\  
&\cdot \sum_{k} \int_{0}^{t} \vert k\cdot (\mathcal{F} _{x,v} G) (t,s,l-k,k(t-s))\vert  ds),
\end{split} 
\end{equation*}
and by integrating in time, we obtain that 
\begin{equation*}
\begin{split}
\norm{K_{G}}^{2}_{L^{2}([0,T],L^{2}_{x})}&\lesssim \sum_{l}  \int_{0}^{T} \int_{0}^{t}\sum_{k} \vert\hat{F_{k}}(s)\vert^{2}   \vert k \cdot (\mathcal{F}_{x,v} G)(t,s,l-k,k(t-s)) \vert ds dt\\  
&\cdot \underset{l}{\sup}\underset{t\in[0,T]}{\sup} \int_{0}^{t}\sum_{k} \vert k\cdot (\mathcal{F} _{x,v} G) (t,s,l-k,k(t-s))\vert  ds) \leq I\cdot II.
\end{split} 
\end{equation*}
Let us first consider the term $II$. We observe that for all $s_1\geq0$ , 
\begin{equation*}
 \underset{l}{\sup}\underset{t\in[0,T]}{\sup} \int_{0}^{t}\sum_{k} \vert k\cdot (\mathcal{F} _{x,v} G) (t,s,l-k,k(t-s))\vert  ds\leq
\end{equation*}
\begin{equation*}
 \underset{l}{\sup}\underset{t\in[0,T]}{\sup} \sum_{k} \left(\underset{0\leq s\leq t}{\sup}\underset{\xi}{\sup}\left[(1+\vert\xi\vert)^{s1}\vert (\mathcal{F} _{x,v} G) (t,s,l-k,\xi)\vert\right] \int_{0}^{t} \frac{\vert k\vert }{(1+\vert k\vert (t-s))^{s_{1}}} ds \right).
\end{equation*}
We choose $s_{1}>1$. Changing variables, we observe that 
\begin{equation*}
\int_{0}^{t}  \frac{\vert k\vert }{(1+\vert k\vert (t-s))^{s_{1}}} ds \leq \int_{0}^{\infty} \frac{1}{(1+\tau^{s_{1}})}d\tau < \infty.
\end{equation*}
It follows that 
\begin{equation*}
II \lesssim  \underset{l}{\sup}\underset{t\in[0,T]}{\sup} \sum_{k} \underset{0\leq s\leq t}{\sup}\underset{\xi}{\sup}\left[(1+\vert\xi\vert)^{s1}\vert (\mathcal{F} _{x,v} G) (t,s,l-k,\xi)\vert\right].
\end{equation*}
Next we choose $s_{2}> d/2$ so that $\sum \frac{1}{(1+\vert k\vert )^{2\,s_{2}}}< \infty$. By using  Cauchy-Schwarz, we deduce 
\begin{equation*}
\begin{split}
II&\lesssim \underset{t\in[0,T]}{\sup} \left(\sum_{k} \underset{0\leq s\leq t}{\sup}\underset{\xi}{\sup}\left((1+\vert k\vert)^{s_{2}}(1+\vert\xi\vert)^{s_{1}}\vert (\mathcal{F} _{x,v} G) (t,s,l-k,\xi)\vert  \right)^{2}\right)^{\frac{1}{2}}\\&\lesssim  \norm{G}_{T,s_{1},s_{2}}.
\end{split}
\end{equation*}

Now let us consider the other term $I$. By using Fubini, we have for all $\alpha\in (0,1)$,
\begin{align*}
 &\sum_{l}  \int_{0}^{T} \int_{0}^{t}\sum_{k} \vert\hat{F_{k}}(s)\vert^{2}   \vert k \cdot (\mathcal{F}_{x,v} G)(t,s,l-k,k(t-s)) \vert ds dt\\
&= \int_{0}^{T}\sum_{k} \vert\hat{F_{k}}(s)\vert^{2} \int_{s}^{T}  \sum_{l}   \vert k \vert \vert (\mathcal{F}_{x,v} G)(t,s,l-k,k(t-s)) \vert dt ds\\
&=  \int_{0}^{T}\sum_{k} \vert\hat{F_{k}}(s)\vert^{2} (1+\vert k\vert^{\alpha}) \int_{s}^{T}  \sum_{l}   \frac{\vert k \vert}{(1+\vert k\vert^{\alpha})} \vert (\mathcal{F}_{x,v} G)(t,s,l-k,k(t-s)) \vert dt ds\\
&= \norm{F}_{L^{2}([0,T],H_x^{\alpha})}^{2} \underset{k}{\sup}\underset{0\leq s \leq T}{\sup}\int_{s}^{T}  \sum_{l}   \frac{\vert k \vert}{(1+\vert k\vert^{\alpha})} \vert (\mathcal{F}_{x,v} G)(t,s,l-k,k(t-s)) \vert  dt.\\
\end{align*}
As previously, by choosing $s_{1}>1$, $s_{2}>d/2$, we have

\begin{align*}
&\underset{k}{\sup}\underset{0\leq s \leq T}{\sup}\int_{s}^{T}  \sum_{l}   \frac{\vert k \vert}{(1+\vert k\vert^{\alpha})} \vert (\mathcal{F}_{x,v} G)(t,s,l-k,k(t-s)) \vert dt \\
&\leq \underset{k}{\sup}\underset{0\leq s \leq T}{\sup}\int_{s}^{T}        \frac{\vert k \vert}{(1+\vert k\vert^{\alpha})(1+\vert k\vert(t-s))^{s_{1}}}\sum_{l} \underset{\xi}{\sup}(1+\vert\xi\vert)^{s1} \vert (\mathcal{F}_{x,v} G)(t,s,l-k,\xi) \vert dt \\
&\lesssim \underset{k}{\sup}\underset{0\leq s \leq T}{\sup}\int_{s}^{T}    \frac{\vert k \vert}{(1+\vert k\vert^{\alpha})(1+\vert k\vert(t-s))^{s_1}}   \norm{G}_{T,s_{1},s_{2}}.
\end{align*}
To conclude, we have to estimate $\underset{k}{\sup}\underset{0\leq s \leq T}{\sup}\int_{s}^{T}   \frac{\vert k \vert}{(1+\vert k\vert^{\alpha})(1+\vert k\vert(t-s))^{s_{1}}}dt$. We first use a change of variable 

\begin{align*}
&\underset{k}{\sup}\underset{0\leq s \leq T}{\sup}\int_{s}^{T}   \frac{\vert k \vert}{(1+\vert k\vert^{\alpha})(1+\vert k\vert(t-s))^{s_{1}}}dt \\
&\leq \underset{k}{\sup}\frac{1}{(1+\vert k\vert^{\alpha})}\underset{0\leq s \leq T}{\sup}\int_{0}^{\vert k\vert (T-s)} \frac {1}{(1+\tau)^{s_{1}}}d\tau\\
&\leq \underset{k}{\sup}\frac{1}{(1+\vert k\vert^{\alpha})}\int_{0}^{\vert k\vert T} \frac {1}{(1+\tau)^{s_{1}}}d\tau.
\end{align*}
Then, by Holder inequality, we have the estimate 
\begin{align*}
& \underset{k}{\sup}\frac{1}{(1+\vert k\vert^{\alpha})}\int_{0}^{\vert k\vert T} \frac {1}{(1+\tau)^{s_{1}}}d\tau\\
&\leq  \underset{k}{\sup}\frac{1}{(1+\vert k\vert^{\alpha})} \times (\vert k\vert T)^{\alpha} \left(\int_{0}^{\infty} \frac{1}{(1+\tau)^{\frac{s_{1}}{1-\alpha}}}d\tau\right)^{1-\alpha}\\
&\lesssim T^{\alpha}.
\end{align*}
We have finally shown that
\begin{equation*}
I\lesssim  T^{\alpha} \norm{G}_{T,s_{1},s_{2}} \norm{F}_{L^{2}([0,T],H_x^{\alpha})}^{2},
\end{equation*}
which ends the proof.
\end{proof}

\subsection{Conclusion of the estimates for $\rho_{\varepsilon}$}

By Lemma \ref{AOP} 
\begin{equation*}
\partial^{I}_{x}\rho_{\varepsilon} =\sum_{K\in\{1,...,d\}^{m}} K_{H_{(K,0),(I,0)}}\left(U*\left(\left(I-\varepsilon^{2}\Delta\right)^{-1}\partial_{x}^{K}\rho\right)\right)+\mathcal{R}_{I,0},
\end{equation*}
with 
\begin{equation*}
\norm{\mathcal{R}_{I,0}}_{L^{2}([0,T],L^{2}_{x})}\lesssim T^{\frac{1}{2}} \Lambda (T,R). 
\end{equation*}
Let $\alpha>0$ be given by Assumption (A1). Notice that if $\alpha\geq 1$, Assumption (A1) will be satisfied for any real inferior to $\alpha$, thus we can assume in the following that $\alpha<1$. \\ 
We get, thanks to Proposition \ref{MdK} 
\begin{align*}
&\norm{K_{H_{(K,0),(I,0)}}\left(U*\left(\left(I-\varepsilon^{2}\Delta\right)^{-1}\partial_{x}^{K}\rho\right)\right)}_{L^{2}([0,T],L_x^{2})}\\
&\lesssim T^{\alpha/2} \norm{\left(U*\left(\left(I-\varepsilon^{2}\Delta\right)^{-1}\partial_{x}^{K}\rho\right)\right)}_{L^{2}([0,T],H_x^{\alpha})}\\
&\lesssim T^{\alpha/2} R.
\end{align*}

We emphasize that this inequality is independent of $\varepsilon$. We can thus estimate all of the $\partial^{I}_{x}\rho_{\varepsilon}$ with $\vert I \vert \leq m$. Summing up all of the results we have, we have proved that 
\begin{equation*}
\norm{\rho_{\varepsilon}}_{L^{2}([0,T],H_x^{m})}\leq T^{\alpha/2}\Lambda(T,R).
\end{equation*}

\section{End of the proof of Theorem 2}

The previous sections have given all the tools to end the proof. Indeed, Section 3 has been dedicated to show that 

\begin{equation*}
\norm{\rho_{\varepsilon}}_{L^{2}([0,T],H_x^{m})} \leq T^{\frac{\alpha}{2}}\Lambda(T,R).
\end{equation*}
We have shown at the end of Section 2 that 
\begin{equation*}
\norm{f_{\varepsilon}}_{L^{\infty}([0,T],\mathcal{H}^{m-1}_{2r})}\leq \norm{f_{\varepsilon}^{0}}_{\mathcal{H}^{m-1}_{2r}} +T^{\frac{1}{2}}\Lambda(T,R), 
\end{equation*}
from which we deduce that for all $0<T< T^{\varepsilon}$
\begin{equation*}
\mathcal{N}_{m,2r} (T,f_{\varepsilon}) \leq \norm{f_{\varepsilon}^{0}}_{\mathcal{H}^{m-1}_{2r}} +T^{\frac{1}{2}}\Lambda(T,R)  +T^{\frac{\alpha}{2}}\Lambda(T,R).
\end{equation*}
Next, we consider $R$ large enough such that 
\begin{equation*}
\frac{1}{2}R> \norm{f_{\varepsilon}^{0}}_{\mathcal{H}^{m-1}_{2r}}.
\end{equation*}
With $R$ being fixed, we can find by continuity a $T^{\#}$ small enough such that the previous estimates are satisfied and that for every $T\in [0,T^{\#}]$, we have
\begin{equation*}
T^{\frac{1}{2}}\Lambda(T,R)  +T^{\frac{\alpha}{2}}\Lambda(T,R) < \frac{1}{2}R.
\end{equation*}
Thus, for every $T\in [0,T^{\#}]$, $\mathcal{N}_{m,2r}(t,f_{\varepsilon}) <R$ and so $T^{\varepsilon}>T^{\#}$. This means that the time $T^{\varepsilon}$ is uniformly bounded from below by a certain $T^{\#}$ which is independent of $\varepsilon$. That concludes the proof of Theorem \ref{theorem2}.

\section{Solutions for the limit problem}

Let us recall the limit system, which is the system we want to study in Theorem \ref{MT}  

\begin{equation*}
 \left\{
    \begin{array}{ll}
        \partial_t{f}+v\cdot \nabla_x f + E  \cdot \nabla_v f  = 0,\\
        E =-\nabla_x \int_{\mathbb{T}^d} U(x-y) \left( \int_{\mathbb{R}^d}f(t,y,v)dv\right)dy , \\
	f (0,x,v) = f^0 (x,v).
    \end{array}
\right.
\end{equation*}
This system is a non linear transport equation. But taking the formal limit $\varepsilon \rightarrow 0$, we have lost the elliptic propriety that we had for the system (\ref{SystP}). Proving local well-posedness is thus more challenging, and we can use Theorem $\ref{theorem2}$ to find a solution of (\ref{SystRef}).
\bigbreak
Indeed, this theorem gives us a family of functions $f_{\varepsilon}$ solutions of (\ref{SystP}) on a time interval $[0,T]$, $T$ being independent of $\varepsilon$. We can use this family of functions to find a solution of the limit system. In the following, we will extract a subsequence that will converge to a certain function $f$. We next have to show that $f$ is actually a solution of the limit system. We will start by showing the uniqueness, by similar arguments that we used in Theorem \ref{theorem2}.

\subsection{Uniqueness of the solution}
\begin{prop}\label{Unicity}
Let $f_{1}$, $f_{2} \in \mathcal{C} ([0,T],\mathcal{H}^{m-1}_{2r})$ with $m>m_0$ and $2r>r_0$ be two solutions of (\ref{SystRef}) with the same initial condition $f^{0}$. We write $\rho_i=\int f_i dv$ and we suppose that $\rho_i \in L^{2}([0,T],H_x^{m})$. Then $f_1 =f_2$ on $[0,T] \times \mathbb{T}^{d}  \times \mathbb{R}^{d}$.
\end{prop}
\begin{proof}
Let $f=f_1-f_2$. We denote by $E_2$ the force field associated with the density $\rho_2$ and $E$ associated with $\rho=\rho_1-\rho_2$. We have that $f$ solves the equation 
\begin{equation}\label{equaf}
\partial_t f+v\cdot \nabla _x f +E\cdot \nabla_v f_1 +E_2\cdot \nabla_v f =0.
\end{equation}
Let $\Phi$ satisfying the Burgers equation (\ref{Burgers}) associated with $E_2$, with initial condition $\Phi(0,x,v)=v$. As previously, we write $g(t,x,\Phi(t,x,v))=f(t,x,v)$. We obtain that $g$ is solution of 
\begin{equation*}
\partial_t g +\Phi \cdot\nabla_x g +E\cdot \nabla_v f_1=0.
\end{equation*}
Recall that

\begin{equation*}
\int_{\mathbb{R}^{d}}g(t,x,v)\mathcal{J}(t,x,v) dv=\rho (t,x),
\end{equation*}
for $\mathcal{J}(t,x,v)= \vert \det\nabla_{v}\Phi(t,x,v)\vert$.\\
We consider the characteristics $X$ satisfying 
\begin{equation*}
\partial_{t}X(t,s,x,v) =\Phi (t,X(t,s,x,v),v), \quad X(s,s,x,v)=x.
\end{equation*}
Following the steps of Lemma \ref{AOP}, we have

\begin{equation*}
\rho (t,x) = K_H (U*\rho),
\end{equation*}
with $H(t,s,x,v) = (\nabla_v f_1)(s,x-(t-s)v,\Psi(t,s,x,v))\mathcal{J}(t,x,\Psi(t,s,x,v))\tilde{\mathcal{J}}(t,s,x,v)$.

We use proposition \ref{MdK} to show that 
\begin{equation*}
\begin{split}
\norm{\rho}_{L^{2}([0,T],L^{2}_x)}&\leq C T^{\alpha/2} \norm{(U*\rho)}_{L^{2}([0,T],H^{\alpha}_x)}
\\&\leq C T^{\alpha/2} \norm{\rho}_{L^{2}([0,T],L^{2}_x)}.
\end{split}
\end{equation*}
We take $T_0$ such that $C T_0^{\alpha/2} < 1 $, and deduce that we must have $\rho=0$ on $[0,T_0]$. Next, we go back to (\ref{equaf}) on $[0,T_0]$
\begin{equation*}
\partial_t f+v\cdot \nabla _x f +E_2\cdot \nabla_v f =0.
\end{equation*}
$f$ is solution of an homogeneous transport equation, with initial condition being $0$. That means that $f=0$ on all $[0,T_{0}]$ . To obtain this result on $[0,T]$, we make the same reasoning on $[T_{0},T]$ to obtain the result on $[0,2T_0]$ and so on.

\end{proof}

\subsection{Existence of the solution}

\bigbreak

All is left to do to prove Theorem \ref{MT} is to show the existence of solutions.\\
Let ($f_{\varepsilon}$) be the family of functions solutions of (\ref{SystP}) with the initial conditions $f^{0}_{\varepsilon}=f^0$. Thanks to Theorem \ref{theorem2}, there exist $T$ and $R$ independent of $\varepsilon$ such that $f_{\varepsilon}\in\mathcal{C}([0,T],\mathcal{H}^{m}_{2r})$ satisfies 
\begin{equation}\label{Major}
\underset{\varepsilon\in (0,1]}{sup}\mathcal{N}_{m,2r}(T,f_{\varepsilon})\leq R.
\end{equation}
We get by (\ref{Major}) that $f_{\varepsilon}$ is uniformly bounded in $\mathcal{C}([0,T],\mathcal{H}^{m-1}_{2r})$. By (\ref{SystRef}), we obtain that $\partial_t  f_{\varepsilon}$ is uniformly bounded in $L^{\infty}([0,T],\mathcal{H}^{m-2}_{2r-1})$. The Ascoli theorem gives the existence of a function $f \in \mathcal{C}([0,T],L^{2}_{x,v})$ and a sequence $\varepsilon_n$ such that $f_{\varepsilon_n}$ converges to $f$ in $\mathcal{C}([0,T],L^{2}_{x,v})$. By interpolation, we actually have convergence in $\mathcal{C}([0,T],\mathcal{H}^{m-1-\delta}_{2r-\delta})$ for all $\delta>0$. By Sobolev embedding,  $f_{\varepsilon_n}$ converges to $f$ in $L^{\infty}([0,T]\times \mathbb{T}^{d}\times \mathbb{R}^{d})$ and $\rho_{\varepsilon_n}$ converges to $\rho=\int f dv$ in $L^2 ([0,T],L_x^2) \cap L^{\infty}([0,T]\times \mathbb{R}^{d})$. Thus, the limit function $f$ solves the system (\ref{SystRef}). \bigbreak
To conclude, we want to apply Proposition \ref{Unicity}. We only have to show that $f \in \mathcal{C}([0,T],\mathcal{H}^{m-1}_{2r})$ and that $\rho\in L^{2}([0,T],H_x^{m})$. By weak compacity arguments,  $f\in L^{\infty}([0,T],\mathcal{H}^{m-1}_{2r})$ and $\rho \in L^{2}([0,T],H_x^m)$. We then use an energy estimate which have been shown previously (the formula (\ref{energy})) to show that 
\begin{equation*}
\frac{d}{dt}\norm{f}_{\mathcal{H}^{m-1}_{2r}}^{2}\leq C.
\end{equation*}
It follows that $f\in \mathcal{C}([0,T],\mathcal{H}^{m-1}_{2r})$. We can finally apply Proposition \ref{Unicity} which gives the uniqueness of $f$ and concludes the proof of Theorem \ref{MT}. 

\bigbreak

\newpage
\newpage

\end{document}